\documentclass[10pt]{amsart}

\usepackage{amsmath}
\usepackage{graphicx}
\usepackage{amsfonts}
\usepackage{amssymb}

\newtheorem{theorem}{Theorem}[section]
\theoremstyle{plain}

\newtheorem{lemma}{Lemma}[section]

\numberwithin{equation}{section}

\begin{document}
\title[Rigidity]{Rigidity of complete entire self-shrinking solutions to K\"{a}hler-Ricci flow}
\author{Gregory Drugan}
\address{Department of Mathematics, Box 354350\\
University of Washington\\
Seattle, WA 98195}
\email{drugan@math.washington.edu}
\author{Peng LU}
\address{Department of Mathematics\\
University of Oregon\\
Eugene, OR 97403}
\email{penglu@uoregon.edu}
\author{Yu YUAN}
\address{Department of Mathematics, Box 354350\\
University of Washington\\
Seattle, WA 98195}
\email{yuan@math.washington.edu}

\thanks{This work was supported by the National Science Foundation [RTG 0838212 to G.D., DMS 1100966 to Y.Y.]; and the Simons Foundation [229727 to P.L.].}

\begin{abstract}
We show that every complete entire self-shrinking solution on complex Euclidean space to the K\"{a}hler-Ricci flow must be generated from a quadratic potential.
\end{abstract}

\maketitle

\section{Introduction}

In this note, we prove the following result.

\begin{theorem}
\label{thm:main}
Suppose $u$ is an entire smooth pluri-subharmonic solution on ${\mathbb{C}}^{m}$ to the complex Monge-Amp\`{e}re equation
\begin{equation}
\label{eq:cma}
\ln\det(u_{\alpha\bar{\beta}})=\frac{1}{2}x\cdot Du-u.
\end{equation}
Assume the corresponding K\"{a}hler metric $g= ( u_{\alpha\bar{\beta}} )$ is complete. Then $u$ is quadratic.
\end{theorem}

Any solution to~(\ref{eq:cma}) leads to a self-shrinking solution $v(x,t)  =-tu( x/\sqrt{-t} )$ to a parabolic complex Monge-Amp\`{e}re equation
\[
v_{t}=\ln\det(v_{\alpha\bar{\beta}})
\]
in $\mathbb{C}^{m} \times (-\infty,0),$ where $z^{\alpha}=x^{\alpha}+\sqrt{-1}x^{m+\alpha}.$ Note that the above equation of $v$ is the potential equation of the K\"{a}hler-Ricci flow $\partial_{t}g_{\alpha\bar{\beta}}=-R_{\alpha\bar{\beta}}.$ In fact, the corresponding metric $( u_{\alpha\bar{\beta}} )$ is a shrinking K\"{a}hler-Ricci (non-gradient) soliton.

Assuming a certain inverse quadratic decay of the metric--a specific completeness assumption--Theorem \ref{thm:main} has been proved in \cite{CCY}. Similar rigidity results for self-shrinking solutions to Lagrangian mean curvature flows were obtained in \cite{HW}, \cite{CCY}, and \cite{DX}.

The idea of our argument, as in \cite{CCY}, is still to force the phase $\ln\det(u_{\alpha\bar{\beta}})$ in equation (\ref{eq:cma}) to attain its global maximum at a finite point. As this phase satisfies an elliptic equation without the zeroth order terms (see (\ref{eq:phase}) below), the strong maximum principle implies the constancy of the phase. Consequently, the homogeneity of the self-similar terms on the right hand side of equation~(\ref{eq:cma}) leads to the quadratic conclusion for the solution.

However, the difficulty of the above argument lies in the first step: Here we cannot construct a barrier as in \cite{CCY}, which requires the specific inverse quadratic decay of the metric, to show the phase achieves its maximum at a finite point. The new observation is that the radial derivative of the phase, which is the negative of the scalar curvature of the metric~(\ref{eq:radial}), is in fact nonpositive; hence the phase value at the origin is its global maximum. The nonnegativity of the scalar curvature is a result of B.-L. Chen \cite{BLChen}, as the induced metric $g(x,t) = \left( u_{\alpha\bar{\beta}} (x/\sqrt{-t}) \right)$ is a complete ancient solution to the (K\"{a}hler-)Ricci flow. Here we provide a direct elliptic argument for the nonnegativity of the scalar curvature for the complete self-shrinking solutions (in Section 3, after necessary preparation in Section 2, where a pointwise approach to Perelman's upper bound of the Laplacian of the distance \cite{P} is also included). Heuristically one sees the minimum of the scalar curvature is nonnegative from its inequality~(\ref{ineq:R}); it is definitely so if the minimum is attained at a finite point. Note that a thorough study of the lower bound of scalar
curvatures of the gradient Ricci solitons is presented in~\cite[Chap.27]{BookRicci4}.

\section{Preliminary results}

For the potential $u$ of the K\"{a}hler metric $g= ( g_{\alpha\bar{\beta}} ) = (u_{\alpha\bar{\beta}})$ on ${\mathbb{C}}^{m}$, we denote the phase by $\Phi=\ln\det(u_{\alpha\bar{\beta}})$. Then the Ricci curvature is given by $R_{\alpha\bar{\beta}}=-\frac{\partial^{2}\Phi}{\partial z^{\alpha}\partial\bar{z}^{\beta}}.$ The ``complex'' scalar curvature is $R=g^{\alpha\bar{\beta}} R_{\alpha\bar{\beta}}$ ($R$ is one-half of the usual ``real'' scalar curvature). Let $\rho(x)$ denote the Riemannian distance from $x$ to $0$ in $({\mathbb{C}}^{m},g)$. For a solution $u$ of~(\ref{eq:cma}), we derive the following equations and inequalities for those geometric quantities.

\subsection{Equation for phase $\Phi$}

Since $u$ is a solution of~(\ref{eq:cma}), the phase satisfies the equation $\Phi=\frac{1}{2}x\cdot Du-u$. Taking two derivatives,
\begin{equation}
\label{ineq:-Ric}
-R_{\alpha\bar{\beta}}=\frac{\partial^{2}\Phi}{\partial z^{\alpha}\partial \bar{z}^{\beta}}=\frac{1}{2}x\cdot Du_{\alpha\bar{\beta}}.
\end{equation}
Differentiating $\Phi=\ln\det(u_{\alpha\bar{\beta}})$,
\[
D\Phi=g^{\alpha\bar{\beta}}Du_{\alpha\bar{\beta}}.
\]
Combining these equations, we get
\begin{equation}
\label{eq:phase}
g^{\alpha\bar{\beta}}\frac{\partial^{2}\Phi}{\partial z^{\alpha} \partial \bar{z}^{\beta}}=\frac{1}{2}x\cdot D\Phi. 
\end{equation}
In particular, we have the important relation
\begin{equation}
\label{eq:radial}
R=-\frac{1}{2}x\cdot D\Phi.
\end{equation}

\subsection{Inequality for scalar curvature $R$}

Differentiating $R=-\frac{1}{2}x\cdot D\Phi$ twice and using $R_{\alpha \bar{\beta}}=-\frac{\partial^{2}\Phi}{\partial z^{\alpha}\partial\bar
{z}^{\beta}}$,
\begin{equation}
\label{eq:Hess of Scalar}
\frac{\partial^{2}R}{\partial z^{\alpha}\partial\bar{z}^{\beta}} = -\frac{\partial^{2}\Phi}{\partial z^{\alpha}\partial\bar{z}^{\beta}}-\frac
{1}{2}x\cdot D\frac{\partial^{2}\Phi}{\partial z^{\alpha}\partial \bar{z}^{\beta}}=R_{\alpha\bar{\beta}}+\frac{1}{2}x\cdot DR_{\alpha\bar{\beta}}.
\end{equation}
Also, differentiating $R=g^{\alpha\bar{\beta}}R_{\alpha\bar{\beta}}$,
\begin{equation}
DR=-g^{\alpha\bar{\gamma}}Du_{\bar{\gamma}\delta}g^{\delta\bar{\beta}}R_{\alpha\bar{\beta}}+g^{\alpha\bar{\beta}}DR_{\alpha\bar{\beta}}. \nonumber
\end{equation}
Hence by~(\ref{ineq:-Ric}),
\begin{align*}
\frac{1}{2}x\cdot DR  &  =-g^{\alpha\bar{\gamma}}\left(  \frac{1}{2}x\cdot Du_{\bar{\gamma}\delta}\right)  g^{\delta\bar{\beta}}R_{\alpha\bar{\beta}}+g^{\alpha\bar{\beta}}\frac{1}{2}x\cdot DR_{\alpha\bar{\beta}} \\
& = g^{\alpha\bar{\gamma}}\left(  R_{\bar{\gamma}\delta}\right)  g^{\delta \bar{\beta}}R_{\alpha\bar{\beta}}+g^{\alpha\bar{\beta}}\frac{1}{2}x\cdot DR_{\alpha\bar{\beta}}.
\end{align*}
Coupled with~(\ref{eq:Hess of Scalar}), we get
\[
g^{\alpha\bar{\beta}}\frac{\partial^{2}R}{\partial z^{\alpha}\partial\bar{z}^{\beta}}-\frac{1}{2}x\cdot DR=R-g^{\alpha\bar{\gamma}}g^{\delta\bar{\beta}}R_{\alpha\bar{\beta}}R_{\bar{\gamma}\delta}\leq R-\frac{1}{m}R^{2},
\]
or equivalently
\begin{equation}
\label{ineq:R}
g^{\alpha\bar{\beta}}\frac{\partial^{2}R}{\partial z^{\alpha}\partial\bar {z}^{\beta}}\leq\frac{1}{2}x\cdot DR+R-\frac{1}{m}R^{2}.
\end{equation}

\subsection{Inequality for distance $\rho$}

Fix a point $x\in{\mathbb{C}}^{m},$ and let $\rho=\rho(x)$. We assume that $x$ is not in the cut locus of $0$. Since $( {\mathbb{C}}^{m},g )$ is complete, there is a (unique) unit speed minimizing geodesic $\chi:[0,\rho]\rightarrow{\mathbb{C}}^{m}$ from $0$ to $x.$ We introduce a vector field $X(\tau)$ along $\chi(\tau)$ defined by $X=\chi^{\alpha}\frac{\partial}{\partial z^{\alpha}}+\chi^{\bar{\beta}}\frac{\partial}{\partial\bar{z}^{\beta}}$, where we regard $\chi \in {\mathbb{C}}^{m}$ as a tangent vector. Note that $X(0)=0$ and $X(\rho) = x^{i}\frac{\partial}{\partial x^{i}}$.

We proceed to compute the directional derivative $x\cdot D\rho(x)$ using the metric $g$:
\begin{align*}
x\cdot D\rho(x) & = \langle X(\rho),\nabla_{g}\rho \rangle_{g} = \langle X(\rho),\dot{\chi} (\rho) \rangle \\
& = \int_{0}^{\rho} \frac{d}{d\tau} \langle X(\tau), \dot{\chi}(\tau) \rangle d\tau = \int_{0}^{\rho} \langle \nabla_{\tau
} X(\tau),\dot{\chi}(\tau) \rangle d\tau,
\end{align*}
where the tangent vector $\dot{\chi}(\tau) = \frac{d}{d\tau}\chi$ and for simplicity of notation we have dropped the subscript $g$ in the inner product $\langle \, , \, \rangle_{g}.$ To calculate the above integrand, we first compute the covariant derivative of $X$ along $\chi$:
\begin{align*}
\nabla_{\tau}X  &  =\dot{\chi}^{\alpha}\frac{\partial}{\partial z^{\alpha}}+\dot{\chi}^{\bar{\beta}}\frac{\partial}{\partial\bar{z}^{\beta}} + \chi^{\alpha}\nabla_{\dot{\chi}}\frac{\partial}{\partial z^{\alpha}} + \chi^{\bar{\beta}}\nabla_{\dot{\chi}}\frac{\partial}{\partial\bar{z}^{\beta}} \\
& = \dot{\chi}+\chi^{\alpha}\Gamma_{\gamma\alpha}^{\mu}\dot{\chi}^{\gamma}\frac{\partial}{\partial z^{\mu}}+\chi^{\bar{\beta}}\Gamma_{\bar{\delta}
\bar{\beta}}^{\bar{\nu}} \dot{\chi}^{\bar{\delta}} \frac{\partial}{\partial\bar{z}^{\nu}}.
\end{align*}
Then using the identity $\Gamma_{\gamma\alpha}^{\mu}g_{\mu\bar{\beta}}=u_{\gamma\alpha\bar{\beta}}$ (for a K\"{a}hler potential) and~(\ref{ineq:-Ric}), we have
\begin{equation}
\langle\nabla_{\tau}X,\dot{\chi}\rangle=1+X\cdot Du_{\alpha\bar{\beta}} \dot{\chi}^{\alpha}\dot{\chi}^{\bar{\beta}}=1-2R_{\alpha\bar{\beta}}\dot{\chi}^{\alpha}\dot{\chi}^{\bar{\beta}}.\nonumber
\end{equation}
Therefore, we have the formula:
\begin{equation}
\label{eq:phase:directional}
x\cdot D\rho(x)=\rho(x)-\int_{0}^{\rho}2R_{\alpha\bar{\beta}}\dot{\chi}^{\alpha}\dot{\chi}^{\bar{\beta}}d\tau.
\end{equation}

We have the following estimate for the Laplacian of the distance function $\rho.$
\begin{lemma}
\label{lemma:dist:ineq}
Suppose $Ric\leq K$ on $B_{g}(0,\rho_{0})$ for $\rho_{0}>0.$ If $\rho(x)>\rho_{0}$ and $x$ is not in the cut locus of $0$, then
\begin{equation}
\label{dist:ineq}
g^{\alpha\bar{\beta}}\frac{\partial^{2}\rho}{\partial z^{\alpha}\partial \bar{z}^{\beta}}(x)\leq\left[  \frac{2m-1}{2\rho_{0}}+\frac{1}{3}K\rho
_{0}\right]  +\frac{1}{2}x\cdot D\rho(x)-\frac{1}{2}\rho(x).
\end{equation}
\end{lemma}

\begin{proof}
The mean curvature $H$ of the geodesic sphere $\partial B_{g}(0,\rho)$ with respect to the normal $\nabla_{g}\rho$ equals $\frac
{-1}{2m-1}\Delta_{g}\rho.$ As calculated in~\cite[p.52]{Calabi}, $H$ satisfies the following differential inequality
\[
H_{\rho}\geq H^{2}+\frac{1}{2m-1}Ric( \nabla_{g}\rho,\nabla_{g} \rho ).
\]
Let $H=\frac{-1}{\rho}+b.$ Since the Riemannian metric $g$ is asymptotically Euclidean as $\rho \to 0,$ we know $b$ is bounded for small $\rho$ (in fact $O(\rho)$). We then have a corresponding inequality for $b$:
\[
\frac{1}{\rho^{2}}+b_{\rho}\geq\frac{1}{\rho^{2}}-2\frac{1}{\rho}b+b^{2} +\frac{1}{2m-1} Ric( \nabla_{g}\rho,\nabla_{g}\rho ),
\]
and consequently
\[
\left(  \rho^{2}b\right)  _{\rho}\geq\rho^{2}b^{2}+\frac{\rho^{2}}{2m-1} Ric( \nabla_{g}\rho,\nabla_{g}\rho ) \geq \frac{\rho^{2}}{2m-1} Ric(\nabla_{g} \rho, \nabla_{g}\rho ).
\]
Integrating along $\chi(\tau)$, we arrive at
\[
b\left( \chi(\rho_{0}) \right)  \geq \frac{1}{\rho_{0}^{2}} \int_{0}^{\rho_{0}}\frac{\tau^{2}}{2m-1} Ric( \dot{\chi},\dot{\chi
} ) d\tau.
\]
Then for $\rho \geq \rho_{0}$,
\begin{align*}
H\left( \chi( \rho ) \right) & = H \left( \chi(\rho_{0}) \right) + \int_{\rho_{0}}^{\rho}H_{\rho}d\tau \\
& \geq H\left( \chi(\rho_{0}) \right) + \int_{\rho_{0}}^{\rho} \frac{1}{2m-1} Ric\left(  \dot{\chi},\dot{\chi}\right) d\tau \\
& \geq\frac{-1}{\rho_{0}}+\frac{1}{\rho_{0}^{2}}\int_{0}^{\rho_{0}}\frac{\tau^{2}}{2m-1} Ric\left( \dot{\chi},\dot{\chi} \right) d\tau+\int_{\rho_{0}}^{\rho}\frac{1}{2m-1} Ric\left( \dot{\chi},\dot{\chi} \right) d\tau.
\end{align*}
Substituting back $\Delta_{g} \rho = -(2m-1) H$ and recalling $x=\chi(\rho)$, we obtain
\begin{equation}
\label{ineq:laplace:distance}
\Delta_{g} \rho \leq \frac{2m-1}{\rho_{0}} - \int_{0}^{\rho} Ric(\dot{\chi},\dot{\chi}) d\tau + \int_{0}^{\rho_{0}} \left(1-\frac{\tau^{2}}{\rho_{0}^{2}} \right) Ric(\dot{\chi},\dot{\chi}) d\tau,
\end{equation}
when $\rho\geq\rho_{0}$. In fact, this estimate was first derived in~\cite[Sec.8]{P} (by a second variation argument).

Note that the Riemannian Laplacian $\Delta_{g}=2g^{\alpha\bar{\beta}} \frac{\partial^{2}}{\partial z^{\alpha}\partial\bar{z}^{\beta}}$ and
$Ric(\dot{\chi},\dot{\chi})=2R_{\alpha\bar{\beta}}\dot{\chi}^{\alpha}\dot{\chi}^{\bar{\beta}}$. Then~(\ref{dist:ineq}) follows from combining~(\ref{eq:phase:directional}) with~(\ref{ineq:laplace:distance}).
\end{proof}

To prove Theorem~\ref{thm:main} we will also need an inequality for $\rho^{y}(x)=$ distance from $x$ to $y$ in $({\mathbb{C}}^{m},g).$ Following the previous argument and using $(X-y^{i}\frac{\partial}{\partial x^{i}})$ instead of $X$ for the vector field along $\chi,$ we have
\begin{lemma}
\label{lemma:dist:ineq:gen}
Suppose $Ric\leq K$ on $B_{g}(y,\rho_{0})$. Fix $A>\rho_{0},$ and let $x\in B_{g}(y,A)$. If $\rho^{y}(x)>\rho_{0}$ and $x$ is
not in the cut locus of $y$, then
\begin{equation}
\label{dist:ineq'}
g^{\alpha\bar{\beta}}\frac{\partial^{2}\rho^{y}}{\partial z^{\alpha} \partial\bar{z}^{\beta}}(x)\leq\left[  \frac{2m-1}{2\rho_{0}}+\frac{1}{3} K\rho_{0}\right]  +\frac{1}{2}x\cdot D\rho^{y}(x)-\frac{1}{2}\rho^{y}(x)+C_{0}A|y|,
\end{equation}
where the constant $C_{0}$ only depends on the ``Euclidean'' norms of $Du_{\alpha\bar{\beta}}$ and $g^{-1}$ in $B_{g}(y,A)$.
\end{lemma}

\begin{proof}
Arguing as above, let $\chi$ be the unit speed minimizing geodesic from $y$ to $x$. Using $(X-y^{i}\frac{\partial}{\partial x^{i}})$ for the vector field along $\chi$ (note that $X(0)=y^{i}\frac{\partial}{\partial x^{i}}$), we have 
\[
(x-y)\cdot D\rho^{y}(x)=\rho^{y}(x)-\int_{0}^{\rho^{y}(x)}2R_{\alpha\bar{\beta}}\dot{\chi}^{\alpha}\dot{\chi}^{\bar{\beta}}d\tau-\int_{0}^{\rho
^{y}(x)}y\cdot Du_{\alpha\bar{\beta}}\dot{\chi}^{\alpha}\dot{\chi}^{\bar{\beta}}d\tau.
\]
The conclusion of the lemma follows, as above, from combining this equation with~(\ref{ineq:laplace:distance}).
\end{proof}

\section{Proof of Theorem~\ref{thm:main}}

First we prove the scalar curvature $R \geq 0$ on complete $({\mathbb{C}}^{m},g).$ Choose a cut-off function $\phi$ such that $\phi\equiv1$ on $[0,1]$, $\phi\equiv0$ on $[2,\infty)$, $\phi^{\prime}\leq0,$ $|\phi^{\prime}|\leq C_{1}\phi^{1/2},$ and $|\phi^{\prime\prime}-2\left(  \phi^{\prime}\right)^{2}/\phi|\leq C_{2}\phi^{1/2}$. For any small $\rho_{0}>0,$ $K=K(\rho_{0})$ can be chosen so that $Ric\leq K$ on $B_{g}(0,2\rho_{0})$. Fixing $A>\rho_{0}$, we derive an effective negative lower bound~(\ref{ineq:Reffective}) for $R$ on $B_{g}(0,A).$ Set $\tilde{R}=\phi(\rho/A)R.$ If $R<0$ at some point in $B_{g}(0,2A)$, then $\tilde{R}$ achieves a negative minimum at some point $p\in B_{g}(0,2A),$ as $\overline{B_{g}(0,2A)}$ is compact for each $A>0$ by the completeness of $({\mathbb{C}}^{m},g).$ We consider two cases.

Case 1: $p$ is not in the cut locus of $0.$ Then $\rho$ is smooth near $p$, and we have
\[
\Delta_{g}\tilde{R}=\left(  \frac{\phi^{\prime\prime}}{A^{2}}+\frac{\phi^{\prime}}{A}\Delta_{g}\rho\right)  R+\phi\Delta_{g}R+2\langle\nabla
\phi,\nabla R\rangle,
\]
where we have used $|\nabla\rho|=1.$ In order to have a linear differential inequality for $\tilde{R}$ with smooth coefficients (even for the Lipschitz function $\rho$), we rewrite
\begin{eqnarray*}
\langle\nabla\phi,\nabla R\rangle & = & \left\langle \nabla\phi,\frac{\nabla\tilde{R}}{\phi}-\frac{\nabla\phi}{\phi^{2}}\tilde{R}\right\rangle \, = \, \left\langle \nabla\frac{\tilde{R}}{R},\frac{\nabla\tilde{R}}{\phi}\right\rangle -\frac{\left\vert \nabla\phi\right\vert ^{2}}{\phi^{2}} \tilde{R} \\
& \overset{|\nabla\rho|=1}{=} & \frac{ |\nabla\tilde{R}|^{2}}{\tilde{R}}-\left\langle \frac{\nabla R}{R},\nabla\tilde{R}\right\rangle -\frac{\left(  \phi^{\prime}\right)  ^{2}}{A^{2}\phi}R \\
& \overset{\tilde{R}<0}{\leq} & -\left\langle \frac{\nabla R}{R},\nabla\tilde{R}\right\rangle -\frac{\left(  \phi^{\prime}\right)^{2}}{A^{2}\phi}R.
\end{eqnarray*}
Using $\Delta_{g}=2g^{\alpha\bar{\beta}}\frac{\partial^{2}}{\partial z^{\alpha}\partial\bar{z}^{\beta}}$, the inequalities~(\ref{dist:ineq}) and
(\ref{ineq:R}) for $\rho$ and $R$, and the inequalities for $\phi$, we get
\begin{eqnarray*}
g^{\alpha\bar{\beta}} \frac{\partial^{2}\tilde{R}}{\partial z^{\alpha}\partial\bar{z}^{\beta}} & \overset{\phi^{\prime}R>0}{\leq} & \frac{1}{2A^{2}}\left[  \phi^{\prime\prime}-\frac{2\left(  \phi^{\prime}\right)  ^{2}}{\phi}\right]  R+\frac{\phi^{\prime}R}{A}\left[  \left(  \frac{2m-1}{2\rho_{0}}+\frac{1}{3}K\rho_{0}\right)  +\frac{1}{2}x\cdot D\rho(x)\right]  \\
& & \, + \,\phi \left(\frac{1}{2}x \cdot DR+R-\frac{1}{m}R^{2} \right) - \left\langle \frac{\nabla R}{R},\nabla\tilde{R} \right\rangle \\
& \leq & \frac{C_{2}}{2A^{2}}\phi^{1/2}\left\vert R\right\vert +\frac{C_{1}}{A}\left( \frac{2m-1}{2\rho_{0}}+\frac{1}{3}K\rho_{0} \right) \phi
^{1/2} |R| -\frac{\phi R^{2}}{m}\\
&  & \, + \, \tilde{R} + \frac{1}{2}x\cdot D\tilde{R} - \left\langle \frac{\nabla R}{R},\nabla\tilde{R}\right\rangle \\
&  \leq & \frac{C(m,\rho_{0})}{A^{2}}+\tilde{R}+b(x)\cdot D\tilde{R},
\end{eqnarray*}
where $b(x)$ is a smooth function and $C(m,\rho_{0})$ is a constant that depends only on $m$, $\rho_{0}$, $C_{1}$, and $C_{2}$. Since $\tilde{R}$
achieves its minimum at $p$, we have $\tilde{R}(p)\geq-\frac{C(m,\rho_{0})}{A^{2}}$ and $R\geq-\frac{C(m,\rho_{0})}{A^{2}}$ on $B_{g}(0,A)$.

Case 2: $p$ is in the cut locus of $0$. Then $\rho$ is not smooth at $p$, and we argue using Calabi's trick \cite[p.53]{Calabi} of approximating $\rho$ from above by smooth functions (cf.~\cite[pp.453-456]{BookRicci1}). For completeness, we include the argument here. Let $\chi$ be a unit speed geodesic from $0$ to $p$ that minimizes length, and define $\rho_{\varepsilon}=\rho^{\chi(\varepsilon)}+\varepsilon$, where $\rho^{\chi(\varepsilon)}$ is the distance to $\chi(\varepsilon)$. Then $\rho_{\varepsilon}(p)=\rho(p)$ and $\rho_{\varepsilon}\geq\rho$ near $p$. Since $p$ is not in the cut locus of $\chi(\varepsilon)$, we know that $\rho_{\varepsilon}$ is smooth near $p$. Let $\tilde{R}_{\varepsilon}=\phi(\rho_{\varepsilon}/A)R$. Then $\tilde{R}_{\varepsilon}$ is smooth near $p$. Furthermore, since $\phi$ is decreasing and $R<0$ near $p$, the above properties of $\rho_{\varepsilon}$ show that $\tilde{R}_{\varepsilon}(p)=\tilde{R}(p)$ and $\tilde{R}_{\varepsilon} \geq\tilde{R}$ near $p$. It follows that $\tilde{R}_{\varepsilon}$ has a local minimum at $p$. Arguing as we did in Case 1, and using Lemma~\ref{lemma:dist:ineq:gen} to estimate $\rho_{\varepsilon}$, we have
\[
g^{\alpha\bar{\beta}}\frac{\partial^{2}\tilde{R}_{\varepsilon}}{\partial z^{\alpha}\partial\bar{z}^{\beta}}\leq\frac{C(m,\rho_{0})}{A^{2}}+\tilde
{R}_{\varepsilon}+b(x)\cdot D\tilde{R}_{\varepsilon}+\frac{\phi^{\prime}R}{A}\left[  C_{0}\ 2A\ |\chi(\varepsilon)|\right],
\]
where $b(x)$ is a smooth function, $C(m,\rho_{0})$ is a constant that depends only on $m$, $\rho_{0}$, $C_{1}$, and $C_{2}$, and $C_{0}$ is a constant depending on the ``Euclidean'' norms of $Du_{\alpha\bar{\beta}}$ and $g^{-1}$ in $\overline{B_{g}(0,2A)}$. Note that we may choose $C_{0}$ independent of (small) $\varepsilon$. At $p$ we have
\[
\tilde{R}_{\varepsilon}(p)\geq-\frac{C(m,\rho_{0})}{A^{2}}-\frac{\phi^{\prime}R}{A}(p)\left[ C_{0}\ 2A\ |\chi(\varepsilon)| \right].
\]
Taking $\varepsilon \to 0$, we arrive at the inequality $\tilde{R}(p)\geq-\frac{C(m,\rho_{0})}{A^{2}}$, which shows $R\geq-\frac{C(m,\rho_{0})}{A^{2}}$ on $B_{g}(0,A)$.

Combining Case 1 and Case 2, we have shown that
\begin{equation}
\label{ineq:Reffective}
R\geq-\frac{C(m,\rho_{0})}{A^{2}} \text{ on } B_{g}(0,A).
\end{equation}
Taking $A\rightarrow\infty$, we arrive at $R\geq0$ on ${\mathbb{C}}^{m}$.

Now, we finish the proof of Theorem~\ref{thm:main}. Since $R\geq0,$ it follows from the equation $R=-g^{\alpha\bar{\beta}}\frac{\partial^{2}\Phi}{\partial z^{\alpha}\partial\bar{z}^{\beta}}=-\frac{1}{2}x\cdot D\Phi$ that $\Phi$ achieves its global maximum at the origin. Applying the strong maximum principle to equation~(\ref{eq:phase}) we conclude that $\Phi$ is constant. Using $\frac{1}{2}x\cdot Du-u=\Phi$, we have
\[
\frac{1}{2}x\cdot D\left[ u + \Phi(0) \right] = u + \Phi(0).
\]

Finally, it follows from Euler's homogeneous function theorem that smooth $u + \Phi(0)$ is a homogeneous order two polynomial.

\bigskip

\textbf{Remark.} In fact, one sees that the\ Lipschitz function $\tilde{R}$ (on the set where $\tilde{R}<0$) is a subsolution to
\[
g^{\alpha\bar{\beta}}\frac{\partial^{2}\tilde{R}}{\partial z^{\alpha} \partial\bar{z}^{\beta}}\leq\frac{C(m,\rho_{0})}{A^{2}}+\tilde{R}+b(x)\cdot
D\tilde{R}
\]
in the viscosity sense by using the same trick of Calabi. It follows from the comparison principle (cf.~\cite[p.18]{CIL}) that the same negative lower bounds for $\tilde{R}$ and hence $R$ can be derived.

\end{document}